\documentclass[12pt]{article}
\usepackage{amsmath,amssymb,amsthm,amsfonts,amscd}
\usepackage{verbatim}

\newtheorem{theorem}{Theorem}[section]
\newtheorem{corollary}{Corollary}
\newtheorem{lemma}[theorem]{Lemma}
\newtheorem{proposition}{Proposition}
\newtheorem{conjecture}{Conjecture}
\newcommand{\R}{\mathbb R}
\newcommand{\eqdef}{\stackrel{\mathrm{def}}{=}}

\begin{document}
\title{Is the Trudinger-Moser nonlinearity a true critical nonlinearity?}

\author{Kyril Tintarev
\\{\small Department of Mathematics}\\{\small Uppsala University}
{\small P.O.Box 480}\\
{\small 751 06 Uppsala, Sweden}\\{\small
kyril.tintarev@math.uu.se}}

\date{}

\maketitle

\begin{abstract} While the critical nonlinearity $\int |u|^{2^*}$ for the
Sobolev space $H^1$ in dimension $N>2$
lacks weak continuity at any point, Trudinger-Moser nonlinearity $\int e^{4\pi
u^2}$ in dimension $N=2$ is weakly continuous at any point except zero. In the
former case the lack of weak continuity can be attributed to invariance with
respect to actions of translations and dilations. The Sobolev space $H_0^1$ of
the unit disk $\mathbb D\subset\R^2$ possesses transformations analogous to
translations (M\"obius transformations) and nonlinear dilations $r\mapsto r^s$.
We present improvements of the Trudinger-Moser inequality with sharper
nonlinearities sharper than $\int e^{4\pi u^2}$, that lack weak continuity at
any point and possess (separately), translation and dilation invariance. We
show, however, that no
nonlinearity of the form $\int F(|x|,u(x))\mathrm{d}x$ is both dilation- and
M\"obius shift-invariant. The
paper also gives a new, very short proof of the conformal-invariant
Trudinger-Moser inequality obtained recently by Mancini and Sandeep
\cite{Sandeep} and of a sharper version of Onofri-type inequality of Beckner
\cite{Beckner}.
\end{abstract}

\section{Introduction}
The classical (Pohozhaev)-Trudinger-Moser inequality
(\cite{Pohozhaev,Trudinger,Moser}) on a bounded domain $\Omega\subset\R^2$,
\begin{equation}
\label{TM}
\sup_{u\in H_0^1(\Omega),\|\nabla u\|_2\le 1}\int_{\Omega} e^{4\pi
u^2}dx<\infty,
\end{equation}
is usually regarded as a natural analog of the Sobolev inequality 
\begin{equation}
\label{Sob}
\sup_{u\in H_0^1(\Omega),\|\nabla u\|_2\le 1}\int_{\Omega}
|u|^{2^*}\mathrm{d}x<\infty,
\end{equation}
where $\Omega\subset\R^N$, $N> 2$, and $2^*=\frac{2N}{N-2}$. 
Indeed, both inequalities correspond to the end points of respective parameter
scales:
replacing the number $4\pi$ in (\ref{TM}) by any $p>4\pi$, or the
number $2^*$ in (\ref{Sob}) by any $q>2^*$, results in the respective
supremum taking the value $+\infty$. When $p<4\pi$, or $q>2^*$, both
nonlinearities become weakly continuous. 

There is a significant difference, however, between the weak continuity
properties of the two functionals at the endpoint value $4\pi$ resp. $2^*$. The
Trudinger-Moser
nonlinearity $\int_{\Omega}
e^{4\pi u^2}$ is weakly continuous at any non-zero point of the ball 
$\{u\in H_0^2(\Omega), \|\nabla u\|_2\le 1\}$ (see \cite{PLL2b}), while 
the functional $\int_\Omega |u|^{2^*}$ lacks weak
continuity at any point. Indeed, assuming, for the sake of simplicity, that
$\Omega$ is a unit ball, and taking a $w\in H_0^1(\Omega)\setminus\{0\}$,
extended by zero to the whole
$\R^N)$, 
we have the sequence
$w_k(x)=2^{\frac{N-2}{2}k}w(2^{k}x)$ 
that weakly converges to zero, while
$\|\nabla w_k\|_2=\|\nabla w\|_2$ and $\|w_k\|_{2^*}=\|w\|_{2^*}$.
Let now $u_k=u+w_k$. From Brezis-Lieb lemma it follows that 
$$
\lim\int_\Omega |u_k|^{2^*}=\int_\Omega |u|^{2^*}+\int_\Omega |w|^{2^*}\neq
\int_\Omega |u|^{2^*}.
$$ 
Since $u_k\rightharpoonup u$, this verifies that the functional $\int_\Omega
|u|^{2^*}$ is not weakly continuous at $u$. 

On the other hand, the Trudinger-Moser nonlinearity is not invariant with
respect to any non-compact semigroup of transformations that we know, that
preserves the gradient norm. 

The subject of this paper is to show that Trudinger-Moser nonlinearity is not a
true critical nonlinearity, in the sense that it is dominated by invariant
nonlinearities that lack weak semicontinuity at
any point. In Section 2 we present such nonlinearity in the radial subspace of
$H_0^1({\mathbb D})$, invariant with respect to
nonlinear dilations. By ${\mathbb D}$ we denote the open unit disk in $\R^2$. In
Section 3 we consider another, M\"obius
shift-invariant functional on $H_0^1(\mathbb D)$, that yields an improved
Trudinger-Moser inequality, and give a new, greatly simplified, proof of the
latter. We also state and prove a related version of Onofri inequality on
$\mathbb D$.
In Section 4 we show that there is no functional
of the form $\int
F(|x|,u)$ that is both dilation- and
M\"obius shift-invariant.

\section{Dilation-invariant nonlinearity}
Let $H_{0,r}^1(\mathbb D)$ denote the subspace of radial functions of 
$H_{0}^1(\mathbb D)$.
The transformations
\begin{equation}
\label{dilations}
h_su(r)\eqdef s^{-\frac12}u(r^s), \, u\in H_{0,r}^1(\mathbb D) \,s>0,
\end{equation}
preserve the norm  $\|\nabla u\|_2$ of $H_{0,r}^1(\mathbb D)$, as well as the
2-dimensional Hardy functional
$\int_{\mathbb D}\frac{u^2}{|x|^2(\log 1/|x|)^2}dx$ (for the Hardy inequality in
dimension 2 see Adimurthi and Sandeep \cite{AdiSandeep} and Adimurthi and
Sekar \cite{AdiSekar}). Furthermore, these transformations preserve the norms of
a family of weighted $L^p$-spaces, $p=[2,\infty]$, analogous to the
weighted-$L^p$ scale with $p\in[2,2^*]$ produced by H\"older inequality in the
case $N>2$, interpolating between the Hardy term $\int\frac{u^2}{|x|^2}dx$ and
the critical nonlinearity $\int |u|^{2^*}dx$.
In the case $N=2$, the critical exponent is formally $2^*=+\infty$ and the
dilation-invariant $L^{2^*}$-norm is
\begin{equation}
\label{infnorm}
\|u\|_{2^*}=\sup_{r\in(0,1)}\frac{|u(r)|}{(2\pi\log\frac{1}{r})^{1/2}}.
\end{equation}
The following statement asserts that the Trudinger-Moser functional is
dominated by the $2^*$-norm.
\begin{proposition}
The functional $\int_{\mathbb D} e^{4\pi u^2}$ on the
set $\{u\in H_{0,r}^1({\mathbb D}), \|u\|_{2^*}<1\}$ is continuous in 
the norm \eqref{infnorm}. 
\end{proposition}
\begin{proof}
From the definition of the $2^*$-norm \eqref{infnorm} it follows that
$e^{4\pi u^2}\le r^{-a}$, where $a=2\|u\|^2_{2^*}<2$. Continuity of
$\int_{\mathbb D}
e^{4\pi u^2}$ is now a consequence of Lebesgue convergence theorem.
\end{proof}
Note that it is well known that the unit ball in the $2^*$-norm contains the
unit ball in the gradient norm. Since the proof is elementary, we provide
this as the following
\begin{lemma}\label{pointwise}
For every $u\in H_{0,r}^1({\mathbb D})$,  
\begin{equation}\label{eq:pointwise}
2\pi|u(r)|^2\le \|\nabla u\|^2_2\log\frac{1}{r},\; r\in[0,1]. 
\end{equation}
\end{lemma}
\begin{proof} Use the Newton-Leibniz formula:
$$
|u(\rho)|^2=\left|\int_1^\rho u'(r)dr\right|^2\le \left|\int_1^\rho
u'(r)r^{-1}rdr\right|^2,
$$ 
and apply Cauchy inequality (with respect to the measure $rdr$) to the product
$u'(r)r^{-1}$  
in the right hand side:
$$
|u(\rho)|^2\le \left|\int_0^1
|u'(r)|^2rdr\right| \left|\int_\rho^1
r^{-2}rdr\right|\le \frac{1}{2\pi}\|\nabla u\|_2^2\log\frac{1}{\rho}.
$$ 
\end{proof}
\begin{proposition}
The norm \eqref{infnorm} lacks weak continuity at any $u\in H_{0,r}^1({\mathbb
D})$.
\end{proposition}
\begin{proof} 
Observe that that for every
$u,v\in
C_{0,r}^\infty({\mathbb D}\setminus\{0\})$,
\begin{equation}
\label{max12}
\|u+k^{1/2}v(r^{-k})\|_{2^*}\to
\max\{\|u\|_{2^*},\|v\|_{2^*}\}. 
\end{equation}
Indeed, for all $k$ sufficiently large, the functions
$u$ and $k^{1/2}v(r^{-k})$ have disjoint support.
By density of $C_{0,r}^\infty({\mathbb D}\setminus\{0\})$ in $H_{0,r}^1({\mathbb
D})$, and by
Lemma~\ref{pointwise}, we may extend \eqref{max12} to all
$u,v\in
H_{0,r}^1({\mathbb D})$. If, moreover, $\|v\|_{2^*} > \|u\|_{2^*}$ and
$u_k(r)=u(r)+k^{1/2}v(r^{-k})$, then $u_k\rightharpoonup u$, but
$\|u_k\|_{2^*}$=$\|v\|_{2^*}>\|u\|_{2^*}$. Consequently the map $u\mapsto
\|u\|_{2^*}$ lacks weak continuity at any point. 
\end{proof}

\section{Translation-invariant nonlinearity}
Adopting, for the sake of convenience, the complex numbers notation $z=x_1+ix_2$
for points $(x_1,x_2)$ on ${\mathbb D}$, we consider the following set of
automorphisms of ${\mathbb D}$, known as M\"obius transformations.
\begin{equation}
\label{eta0}
\eta_\zeta (z)= \frac{z-\zeta}{1-\bar\zeta z}, \zeta\in {\mathbb D}.
\end{equation}
Since the maps (\ref{eta0}) are conformal automorphisms of ${\mathbb D}$, one
has $|\nabla u\circ \eta_\zeta|_2=|\nabla u|_2$, which implies that the M\"obius
shifts $u\mapsto u\circ\eta_\zeta$, $\zeta\in {\mathbb D}$, preserve the
gradient norm $\|\nabla u\|_2$. Moreover, they preserve the measure
$\frac{dx}{(1-|x|^2)^2}$. In fact, the gradient norm can be interpreted, under
the Poincar\'e disk model of the hyperbolic space $\mathbb H^2$, as the norm
associated with the Laplace-Beltrami operator on the hyperbolic space $\dot
H^1(\mathbb H^2)$, defined by completion of $C_0^\infty(\mathbb H^2)$, and the
measure $\frac{dx}{(1-|x|^2)^2}$ is the Riemann measure on $\mathbb H^2$.
Moreover, transformations (\ref{eta0}) form a non-compact group of isometries of
$\mathbb H^2$. 

The following inequality (originally expressed in terms of $\mathbb H^2$) has
been shown by Mancini and Sandeep \cite{Sandeep}.
\begin{theorem}
 \label{thm:WTM} Let ${\mathbb D}$ be the open unit disk. The following relation
holds true:
\begin{equation}
\label{WTM}
\sup_{u\in H_0^1({\mathbb D}),\|\nabla u\|_2\le 1}\int_{{\mathbb D}}
\frac{e^{4\pi u^2}-1}{(1-|x|^2)^2}dx<\infty.
\end{equation}
\end{theorem}
Note that the nonlinearity in (\ref{WTM}) dominates the Trudinger-Moser
nonlinearity $e^{4\pi u^2}-1$. 
Furthermore,

\begin{proposition}
The functional  
$$
J(u)=\int_{{\mathbb D}} \frac{e^{4\pi u^2}-1}{(1-|x|^2)^2}dx
$$
lacks weak continuity at any point in $H_0^1({\mathbb D})$.
\end{proposition}
\begin{proof} We give the proof for $u\in C_0^\infty({\mathbb D})$. Extension of
the proof
to general $u\in H_0^1({\mathbb D})$, based on the continuity of $J(u)$ and the
density of
$C_0^\infty({\mathbb D})$ in $H_0^1({\mathbb D})$, is left for the reader. 
let $w\in C_0^\infty({\mathbb D})$, $w\neq 0$, let $\zeta_k=1-1/k$ and define  
$w_k=w\circ\eta_{\zeta_k}$, $u_k=u+w_k$. 
Then $u_k\rightharpoonup u$ and, for $k$ sufficiently large, $u$ and $w$ have
disjoint supports. Therefore, for $k$ large, taking into account M\"obius
shift-invariance of the functional $j$, we have 
$$
J(u_k)=J(u)+J(w_k)=J(u)+J(w)\neq J(u),
$$
and thus $J$ is not weakly continuous at $u$.
\end{proof}
We give now a new proof of Theorem~\ref{thm:WTM}.
\begin{proof} Note that the standard rearrangement argument applies on $\mathbb
H^2$ in an analogous way to that in the Euclidean case, with the Riemannian
measure on $\mathbb H^2$ replacing the Lebesgue measure (see \cite{Baernstein}).
Consequently, it suffices to consider the inequality only for 
radial functions on the Poincar\'e disk. 

Let $u\in H_0^1({\mathbb D})$  be an arbitrary
function satisfying $\|\nabla u\|_2\le
1$. 
We evaluate the integral for $r\le\frac12$ by the standard
Trudinger-Moser inequality. For  $\frac12\le r\le 1$ we estimate the weight in
the
integral by the distance to the boundary: $\frac{1}{(1-r^2)^2}\le
\frac{4}{9}\frac{1}{(1-r)^2}$. Then
\begin{equation}\label{WTM-1}
 \begin{split}
\int_{{\mathbb D}} \frac{e^{4\pi u^2}-1}{(1-r^2)^2}dx 
\le & \frac{16}{9}\int_{{\mathbb D}} e^{4\pi
u^2}dx
 +  2\pi\frac{4}{9}\int_\frac12^1 \frac{e^{4\pi
u^2}-1}{(1-r)^2}rdr\\
\le & \frac{16}{9}\int_{{\mathbb D}} e^{4\pi
u^2}dx + 2\pi \frac{4}{9}\int_\frac12^1\frac{e^{4\pi
u^2}-1}{(1-r)^2}rdr.
 \end{split}
\end{equation}
Let us apply now to the right hand side Lemma~\ref{pointwise} (which gives
$e^{4\pi u^2}\le \frac{1}{r^2}\le 4$ for $r\in[\frac12,1]$), and use the
elementary
inequality $e^t-1\le te^t$ that holds for $t>0$:
\begin{equation} \label{WTM-2}
\begin{split}
2\pi \frac{4}{9}\int_\frac12^1 \frac{e^{4\pi
u^2}-1}{(1-r)^2}rdr 
\;\le \; & 2\pi \frac{4}{9}\int_\frac12^1 \frac{u^2e^{4\pi u^2}}{(1-r)^2}rdr\\
\; \le \; & \frac{16}{9}\int_{{\mathbb D}} \frac{u^2}{(1-r)^2}dx
\le \frac{64}{9}\;.
\end{split} 
\end{equation}
The bound in the right hand side is due the Hardy inequality (with the distance
from the boundary). Thus (\ref{WTM}) follows from
substitution of (\ref{WTM-2}) into (\ref{WTM-1}) and the standard
Trudinger-Moser inequality.
\end{proof}
The argument above is not surprising in the sense that in the higher dimensions
one can derive the Sobolev inequality on $\R^N$ (although not with the
optimal constant) from the Hardy inequality using the pointwise estimate for
the radial functions in $\mathcal D^{1,2}(\R^N)$, 
$\sup |u(r)|r^{\frac{N-2}{2}}\le C\|\nabla u\|_2^2$. This argument also leads to
improvements (without an optimal constant) in Onofri-type inequalities. 
Here we consider an Onofri-type inequality on the unit disk due to Beckner
\cite{Beckner},
\begin{equation}
\label{Onofri}
\log \left(\frac{1}{\pi}\int_{\mathbb
D}e^u\right)+\left(\frac{1}{\pi}\int_{\mathbb D}e^u\right)^{-1}\le
1+\frac{1}{16\pi}\|\nabla u\|_2^2,
u\ge 0.
\end{equation}

\begin{theorem} 
\label{thm:Onofri} There exists a constant $C>0$ such that
for every $u\in H_0^1({\mathbb D})$, $u\ge
0$,
\begin{equation}
 \label{eq:Onofri}
\log \left(\int_{\mathbb D}\frac{e^u-1-u}{(1-r^2)^2}dx\right)\le
C+\frac{1}{16\pi}\|\nabla u\|_2^2.
\end{equation}
\end{theorem}
Note that since we do not know the optimal value of the constant $C$, we do not
have to include the term corresponding to $\left(\frac{1}{\pi}\int_{\mathbb
D}e^u\right)^{-1}$, since it is bounded by $1$. 

\begin{proof} The proof is similar to that of Theorem~\ref{thm:WTM} and we only
sketch the main points. Reduction to the radial functions uses
the same argument as Theorem~\ref{thm:WTM}.  The estimate of the integral over
$r\in[0,\frac12]$ follows immediately from \eqref{Onofri}. To estimate the
integral over
$[\frac12,1)$ note that, on this interval, $\frac{1}{1-r^2}\le
\frac23\frac{1}{1-r}$, and since $u\ge 0$, that $e^u-1-u\le u^2$. Thus we obtain
\begin{equation}
\label{onfr}
\int_\frac12^1\frac{e^u-1-u}{(1-r^2)^2}rdr\le
\frac{4}{9}\int_\frac12^1\frac{u^2e^u}{(1-r)^2}rdr.
\end{equation}
Note now that by Lemma~\ref{pointwise},
$$
u(r)\le (2\pi)^{-\frac12}\|\nabla u\|_2\sqrt{\log \frac{1}{r}}
\le (2\pi)^{-\frac12}\|\nabla u\|_2\sqrt{\log 2}.
$$ 
Thus, for any $\epsilon>0$ there exists $C_\epsilon$ such that for all
$r\in[\frac12,1]$,
$$
u(r) \le \epsilon \|\nabla u\|_2^2 + C_\epsilon. 
$$

Substituting this estimate into the right hand side of \eqref{onfr}, we obtain,
using the usual Hardy inequality with the distance from the boundary,
$$
\int_\frac12^1\frac{e^u-1-u}{(1-r^2)^2}rdr\le
\frac{4}{9} e^{\epsilon \|\nabla
u\|_2^2+C_\epsilon}\int_\frac12^1\frac{u^2}{(1-r)^2}rdr\le
\frac{16}{9\cdot 2\pi}\|\nabla u\|_2^2 e^{\epsilon \|\nabla
u\|_2^2+C_\epsilon}. 
$$
Choosing a suitable $\epsilon$ we conclude that  
$$
\log\left(\int_{\frac12\le |x|<1}\frac{e^u-1-u}{(1-r^2)^2}dx\right)\le
C+2\log(\|\nabla u\|_2)+ \epsilon \|\nabla
u\|_2^2+C_\epsilon\le \frac{\|\nabla
u\|_2^2}{16\pi} + \hat C, 
$$
which, combined with the estimate for the integral over $|x|\le\frac12$, gives
\eqref{eq:Onofri}.
\end{proof}

\section{Non-existence of a perfect critical nonlinearity}

We verify first what invariance requirements have to be satisfied 
by a non-negative function $F$ so that the functional 
$$
J(u)=\int_{\mathbb D} F(|x|,u)dx 
$$
will be invariant with respect to M\"obius shifts or actions
of nonlinear dilations.

\begin{lemma}
\label{lem:tinv}
Let $\eta_\zeta$ be as in (\ref{eta0}) and let 
$F\in C^1((0,\infty)\times\R)$ be a non-negative function. If the functional 
$$
J(u)=\int_{\mathbb D}F(|x|,u)dx 
$$
is continuous on $H_0^1({\mathbb D})$ and
satisfies 
\begin{equation}
\label{tinv}
J(u\circ\eta_\zeta)=J(u) 
\end{equation}
for all $u\in H_0^1({\mathbb D})$ and  $\zeta\in {\mathbb D}$,
then 
\begin{equation}
\label{tinv1}
F(r,u)=\frac{G(u)}{(1-r^2)^2}.
\end{equation}
for some function $G$.
\end{lemma}
\begin{proof} Let us use the complex numbers notation for points in the unit
disk.
Consider (\ref{tinv}) with real-valued $\zeta=t$, that is, with
$\eta_t(x)=\frac{z-t}{1-tz}$.
Assume for the sake of simplicity that $F(r,u)=\frac{G(r^2,u)}{(1-r^2)^2}$.
Then, by invariance of the Riemannian measure on $\mathbb H^2$ with respect to
M\"obius transformations, 
$$
J(u\circ\eta_t)=\int_{\mathbb D} G(|\eta_{-t}z|^2,u(z))dxdy.
$$
Explicit calculation of the derivative gives then, due to \eqref{tinv},
$$
\frac{d}{dt}J(u\circ\eta_t)|_{t=0}=\int_{\mathbb D} 2x(1-r^2)\partial_1
G(r^2,u)\frac{dxdy}{(1-r^2)^2}, 
$$
Restricting now our consideration to those functions $u$ whose support lies in
the right half-disk, we conclude that $\partial_1 G(x^2+y^2,u(x,y))=0$ for
$(x,y)\in {\mathbb D}$, $x>0$, which implies \eqref{tinv1}.
\end{proof}

\begin{lemma} \label{lem:dinv}
Let $h_s$ be as in (\ref{dilations}), and let 
$F\in C^1((0,\infty)\times\R)$ be a non-negative function. If the functional 
$$
J(u)=\int_{\mathbb D} F(|x|,u)dx 
$$
is continuous on $H_0^1({\mathbb D})$ and
satisfies 
\begin{equation}
\label{dinv}
J(h_su)=J(u) 
\end{equation}
for all $u\in H_{0,r}^1({\mathbb D})$ and $s>0$,
then 
\begin{equation}
\label{dinv1}
F(r,u)=H\left(\frac{u}{(\log\frac{1}{r})^\frac12}\right)
\end{equation}
with some function $H$.
\end{lemma}
\begin{proof}
Consider (\ref{dinv}) with radial functions. Assume for the sake of
simplicity that $F(r,u)=H\left(\frac{u}{(\log\frac{1}{\rho})^\frac12}\right)$.
Evaluation of $J(h_su)$ by
the change of variable $r=\rho^\frac{1}{s}$ gives
\begin{equation*}
\begin{split}
J(h_su)=&\\
& \int
s^{-1}H\left(\rho^\frac{1}{s},\frac{u(\rho)}{(\log\frac{1}{\rho})^\frac12}
\right)
\rho^{2/s-2}\rho d\rho.
\end{split} 
\end{equation*}

Evaluating $\frac{dJ(h_su)}{ds}$ at $s=1$ and using the argument analogous
to that in Lemma~\ref{lem:tinv}, we conclude that $\partial_1 H=0$ and
\eqref{dinv1} follows. 
\end{proof}

An immediate conclusion of Lemmas~\ref{lem:tinv} and~\ref{lem:dinv} follows.

\begin{corollary}
Let $F\in C^1((0,\infty)\times\R)$ be a non-negative function satisfying both
invariance requirements \eqref{tinv} and \eqref{dinv}. Then F=0. 
\end{corollary}

This statement does not mean that in the two-dimensional Sobolev space there is
no perfectly critical nonlinearity - that is, that there is no continuous
invariant functional that dominates the
Trudinger-Moser functional and lacks weak continuity at any point.
It means merely that one cannot find a nontrivial functional with required
prthatoperties that has
the form $\int_{\mathbb D} F(|x|,u)$. We still can postulate the following

\begin{conjecture} There is a continuous convex functional $J(u)$ on
$H_0^1({\mathbb D})$, bounded for $\|\nabla u\|_2\le 1$, and satisfying
the following requirements:
\begin{itemize}
 \item[(a)] $J(u\circ\eta_\zeta)=J(u)$ for all $\zeta \in {\mathbb D}$,
\item[(b)] $J(u\circ h_s)=J(u)$ for all $s>0$ and radial $u$,
\item[(c)] the functional $J$ lacks weak continuity at any point,
\item[(d)] the functional $J$ induces an
Orlicz space where Trudinger-Moser functional is continuous and bounded on
every bounded set. 
\end{itemize}
\end{conjecture}
In this conjecture, condition (b) is subject to further interpretation. In
particular, gradient norm-preserving nonlinear dilations can be defined for more
general (but not all) functions in $H_0^1({\mathbb D})$ by the formula
$h_su(z)=s^{-1/2}u(z^s)$. If one drops (b) altogether, this conjecture is
satisfied by the functional \eqref{WTM} of Mancini and Sandeep.

\end{document}